\documentclass[12pt]{amsart}
\usepackage{epsfig,color}

\theoremstyle{definition}

\newtheorem{theorem}{Theorem}[section]

\newtheorem{proposition}[theorem]{Proposition}
\newtheorem{lemma}[theorem]{Lemma}
\newtheorem{remark}[theorem]{Remark}

\numberwithin{equation}{section}

\newcommand{\ee}{{\text{e}}}
\newcommand{\hps}{{\text{hypersurface}}}
\newcommand{\sk}{{\text{self-shrinker}}}

\newcommand{\lb}[1]{\langle#1\rangle}

\newcommand{\mf}{\mathbf}

\title{On the Rigidity of Mean Convex Self-shrinkers}

\author{Qiang Guang}
\address{Department of Mathematics, Massachusetts Institute of Technology, Cambridge, MA  02139, USA}
\email{qguang@math.mit.edu}

\author{Jonathan J. Zhu}
\address{Department of Mathematics,
Harvard University, Cambridge, MA 02138, USA}
\email{jjzhu@math.harvard.edu}

\begin{document}

\begin{abstract}


Self-shrinkers model singularities of the mean curvature flow; they are defined as the special solutions that contract homothetically under the flow. Colding-Ilmanen-Minicozzi showed that cylindrical self-shrinkers $\mf{S}^k\times \mf{R}^{n-k}$ are rigid in a strong sense - that is, any self-shrinker that is mean convex with uniformly bounded curvature on a large, but compact, set must be a round cylinder. Using this result, Colding and Minicozzi were able to establish uniqueness of blowups at cylindrical singularities, and provide a detailed description of the singular set of generic mean curvature flows. 

In this paper, we show that the bounded curvature assumption is unnecessary for the rigidity of the cylinder if either $n\leq 6$, or if the mean curvature is bounded below by a positive constant. These results follow from curvature estimates that we prove for strictly mean convex self-shrinkers. We also obtain a rigidity theorem in all dimensions for graphical self-shrinkers, and curvature estimates for translators of the mean curvature flow. 

\end{abstract}
\maketitle


\section{Introduction}

Mean curvature flow (``MCF") is an evolution equation where a one-parameter family of hypersurfaces $M_t \subset \mf{R}^{n+1}$ flows by mean curvature, that is, it satisfies 
\begin{equation}
(\partial_t x )^\bot = -H\mf{n},
\end{equation}
where $x$ is the position vector, $H$ is the mean curvature and $\mf{n}$ is the outward unit normal. MCF is the negative gradient flow of the area functional. 

We  call a hypersurface $\Sigma^n \subset \mathbf{R}^{n+1}$ a \textit{self-shrinker}, or more simply a \textit{shrinker}, if it satisfies  \begin{equation}\label{eq:shrinker}
H = \frac{1}{2}\langle x,\mathbf{n}\rangle.
\end{equation}
It is easy to see that a self-shrinker is  the $t=-1$ time-slice of a mean curvature flow that evolves by shrinking homothetically to the origin $x=0$. 
The simplest examples in $\mf{R}^{n+1}$ are generalized cylinders $\mathbf{S}^k \times \mathbf{R}^{n-k}$. Here, and henceforth, $\mathbf{S}^k$ denotes the round sphere of radius $\sqrt{2k}$. When $n=1$, the only smooth complete embedded self-shrinkers are straight lines through the origin, and the circle of radius $\sqrt{2}$ (see \cite{AbL}). Therefore, we will always assume $n \geq 2$; in these higher dimensions, there are many more self-shrinkers (see for instance \cite{Ang}, \cite{Chopp} and \cite{KKM}).

By the combined work of Huisken \cite{Hui90}, Ilmanen \cite{I1} and White \cite{W6}, singularities of MCF are modeled by self-shrinkers. As such, one of the most important questions in the study of MCF is to classify the possible singularities. The first major result is due to Huisken (\cite{Hui90}, \cite{Hui93}), who showed that the only smooth complete embedded self-shrinkers in $\mf{R}^{n+1}$ with $H\geq 0$, polynomial volume growth and $|A|$ bounded are generalized cylinders $\mf{S}^k \times \mf{R}^{n-k}$. Later, Colding-Minicozzi \cite{CM1} were able to remove the assumption of bounded curvature $|A|$. Consequently, they showed that the only generic shrinkers are the generalized cylinders $\mf{S}^k \times \mf{R}^{n-k}$, in the sense that all others can be perturbed away. 

In this paper we consider self-shrinkers that are either mean convex or graphical on compact sets. Our first main result is the following local curvature estimate for uniformly mean convex shrinkers:

\begin{theorem}\label{thm:curv2}
Given $n$ and $\delta>0$, there exists $C=C(n, \delta)$ so that for any smooth properly embedded self-shrinker $\Sigma^n \subset \mf{R}^{n+1}$  
which satisfies
\begin{itemize}
\item[$(\star)$] $H \geq \delta$ on $ B_R \cap \Sigma$ for  $R>2$,
\end{itemize}
we have
\begin{equation}
|A|(x) \leq  \frac{CR}{R - |x|} H(x),\,\,\,\,\,\,\, \text{for all}\,\,\,x\in B_{R-1}\cap \Sigma.
\end{equation}
\end{theorem}

The proof of Theorem \ref{thm:curv2} is inspired by the interior curvature estimates of Ecker-Huisken  \cite{EH2}, which give local curvature estimates for MCF with bounded gradient. Similar arguments may also be found in \cite{AG} and \cite{CNS}.

One central problem in the study of the singularities is the uniqueness of blowups, namely, whether different sequences of dilations might give different blowups. For compact singularities of MCF, this uniqueness problem is better understood; see for instance \cite{Sc1} and \cite{Se1}. The first uniqueness theorem for blowups at noncompact singularities was obtained by  Colding-Ilmanen-Minicozzi \cite{CM5}, who proved that if one blowup at a singularity of MCF is a multiplicity-one cylinder, then every subsequential limit is also a cylinder, and Colding-Minicozzi \cite{CM16}, who showed that the axis of the cylinder is also independent of the sequence of rescalings. Using this uniqueness in a fundamental way, Colding-Minicozzi \cite{CM15} were able to give a quite complete description of the singular set for MCF having only generic singularities. 

The key to proving the uniqueness at cylindrical singularities is the rigidity theorem of \cite[Theorem 0.1]{CM5}, which says that any self-shrinker that is mean convex with bounded  $|A|$ on a large compact set must in fact be a cylinder. As an application of Theorem \ref{thm:curv2}, we show that their rigidity theorem holds even without the assumption on $|A|$, so long as $n\leq 6$. Specifically, we prove the following:

\begin{theorem}\label{thm:H>0}
Given $n \leq 6$ and  $\lambda_0$, there exists $R=R(n,\lambda_0)$ so that if $\Sigma^n\subset \mf{R}^{n+1}$ is a self-shrinker with entropy $\lambda (\Sigma) \leq \lambda_0$  which satisfies
\begin{itemize}
\item[($\dagger$)]  $ H \geq 0 $ on $ B_R \cap \Sigma$, 
\end{itemize}
then $\Sigma$ is a generalized cylinder $\mf{S}^k \times \mf{R}^{n-k}$ for some $0\leq k \leq n$.
\end{theorem}

It is important to emphasize that we do not assume any bound for the curvature $|A|$ in the above theorem. In this way Theorem \ref{thm:H>0} is analogous to Colding-Minicozzi's removal of the curvature assumption in Huisken's classification of mean convex shrinkers, and gives a quantitative and stronger version of their result, for $n\leq 6$. This restriction on dimension comes from the curvature estimate for shrinkers with positive mean curvature. Namely, any shrinker in $\mf{R}^{n+1}$, $n\leq 6$, with $H>0$ must satisfy $|A|(x) \leq C(1+ |x|)$ for some $C$ depending only on the volume growth and $n$; see Lemma \ref{lemma:v:u:positive_mean} or Section 3 in \cite{GZ4}.  

If we assume a lower bound of the mean curvature, we obtain the following rigidity theorem that holds in all dimensions.

\begin{theorem} \label{thm:main}
Given $n$, $\lambda_0$ and $\delta>0$, there exists $R=R(n,\lambda_0, \delta)$ so that if $\Sigma^n\subset \mf{R}^{n+1}$ is a self-shrinker with entropy $\lambda (\Sigma) \leq \lambda_0$  which satisfies
\begin{itemize}
\item[($\ddagger$)]  $ H \geq \delta $ on $ B_R \cap \Sigma$, 
\end{itemize}
then $\Sigma$ is a generalized cylinder $\mf{S}^k \times \mf{R}^{n-k}$ for some $1\leq k \leq n$.
\end{theorem}


Finally, we give a rigidity theorem for graphical shrinkers in all dimensions. In \cite{GZ4}, the authors showed that for $n \leq 6$, any shrinker in $\mf{R}^{n+1}$ which is graphical (in the sense that for a constant unit vector $V$, the normal part $\lb{V,\mf{n}}>0$) inside a large, but compact, set must be a hyperplane. Again the restriction on dimension came from the curvature estimate for graphical shrinkers (see Theorem 0.4 in \cite{GZ4}). Here, if we assume that $\lb{V, \mf{n}}$ has a positive lower bound, then in all dimensions we obtain the following theorem as a direct consequence of the curvature estimate of Ecker-Huisken \cite[Theorem 3.1]{EH2} (see also Theorem \ref{thm:curv_graph}) and some ingredients from \cite{GZ4}.

\begin{theorem}\label{thm:graph} Given $n$, $\lambda_0$ and $\delta > 0$, there
  exists $R=R(n,\lambda_0 , \delta)$ so that if $\Sigma^n \subset
  \mf{R}^{n+1}$ is a self-shrinker with entropy $\lambda (\Sigma) \leq
  \lambda_0$ satisfying
  \begin{itemize}
   \item $w=\lb{V,\mf{n}}\geq \delta$ on $B_R \cap \Sigma$ for some constant unit vector $V$,
  \end{itemize}
  then $\Sigma$ is a hyperplane.
\end{theorem}

Let us now briefly outline the structure of this paper. In Section \ref{sec:notation}, we review some key definitions and notation. In Section \ref{sec:curv}, we prove our main curvature estimate Theorem \ref{thm:curv2}. We discuss the proofs of Theorems \ref{thm:H>0} and \ref{thm:main} in Section \ref{sec:rigidity}, by adapting the iteration and improvement scheme of Colding-Ilmanen-Minicozzi \cite{CM5}. We are also able to give shorter proofs using a compactness argument (see Remark \ref{rmk:direct}), but we believe that the improvement method provides a more effective argument. In particular, the cylindrical estimates of Lemma \ref{cor:v:u:taubound_1} may be of independent interest. Finally, in Section \ref{sec:bernstein} we provide the proof of Theorem \ref{thm:graph} as well as some curvature estimates and a Bernstein-type theorem for translators of the mean curvature flow.

\subsection*{Acknowledgements}
The authors would like to thank Professor William Minicozzi for his ever helpful advice and encouragement. The second author is supported in part by the National Science Foundation under grant DMS-1308244. 

\section{Notation and Background}
\label{sec:notation}

\subsection{Notation}
Let $\Sigma^n \subset \mathbf{R}^{n+1}$ be a smooth hypersurface, $\Delta$ its Laplace operator, $A$ its second fundamental form and $H$=div$_{\Sigma}\mathbf{n}$ its mean curvature. We denote by $B_R(x)$ the (closed) ball in $\mathbf{R}^{n+1}$ of radius $R$ centered at $x$. For convenience we will introduce the shorter notation $B_R = B_R(0)$. 

We begin by recalling the following classification of smooth, embedded mean convex self-shrinkers from \cite{CM1}.

 \begin{theorem}(\cite{CM1}) \label{thm:huisken} $\mf{S}^k\times \mf{R}^{n-k}$
    are the only smooth complete embedded self-shrinkers without boundary,
    with polynomial volume growth, and $H \geq 0$ in $\mf{R}^{n+1}$.
  \end{theorem}
We will also consider the operators $\mathcal{L}$ and $L$ from \cite{CM1} defined by   
  
  \begin{equation}
  \mathcal{L}=\Delta-\frac{1}{2}\lb{x, \nabla \cdot},
  \end{equation}
  \begin{equation}
L=\Delta-\frac{1}{2}\lb{x, \nabla \cdot} +|A|^2 + \frac{1}{2}.
\end{equation}

The next lemma records three useful identities from \cite{CM1}.
\begin{lemma}(\cite{CM1})\label{lemma:simons}
If $\Sigma^n \subset \mf{R}^{n+1}$ is a smooth \sk, then for any constant vector $V \in\mathbf{R}^{n+1}$ we have
\begin{equation}
 L H= H,
\end{equation}
\begin{equation}   
L \lb{V,\mf{n}}=\frac{1}{2} \lb{V, \mf{n}}
\end{equation}
and
\begin{equation}
\mathcal{L} |A|^2= |A|^2 - 2|A|^4 +2 |\nabla A|^2.
\end{equation}
\end{lemma}

Colding and Minicozzi \cite{CM1} introduced the entropy $\lambda$ of a \hps\ $\Sigma$, defined as
\begin{equation}
\lambda (\Sigma) = \sup_{x_0,t_0} F_{x_0,t_0}(\Sigma) = \sup_{x_0,t_0}\  (4\pi t_0)^{-\frac{n}{2}} \int_{\Sigma} {\ee^{-\frac{|x-x_0|^2}{4t_0}}} d\mu,
\end{equation}
where the supremum is taking over all $t_0 >0$ and $x_0 \in \mathbf{R}^{n+1}$. It was proven in \cite{CM1} that for a self-shrinker, the entropy is achieved by the $F$-functional $F_{0,1}$, so no supremum is needed. Note that Cheng and Zhou \cite{CZ} (see also \cite{DX3}) proved that for self-shrinkers,  finite entropy, polynomial volume growth and properness are all equivalent.

\section{Curvature estimates for strictly mean convex shrinkers} \label{sec:curv}
This section is devoted to proving Theorem \ref{thm:curv2}. The proof requires some modifications of Ecker-Huisken's interior estimates for mean curvature flow \cite{EH2} (see also \cite{Eck}), in which the authors derive curvature estimates using the maximum principle under the assumption that the flow is locally graphical. For our estimates, the mean convexity will replace the local graphical assumption - in particular, the key ingredient is the identity $LH=H$ that holds on all shrinkers.

First, in order to apply the maximum principle, we describe the choice of cutoff functions and detail the relevant computations:

Fix $n$ and $\delta>0$. Let $\Sigma^n \subset \mf{R}^{n+1}$ be a self-shrinker which satisfies 
\begin{itemize}
\item[$(\star)$] $H \geq \delta$ on $ B_R \cap \Sigma$ for  $R>2$.
\end{itemize}
Set $v=1/H$ and $v_0=1/\delta$, then we have $v\leq v_0$ on  $B_R \cap \Sigma$. Lemma \ref{lemma:simons} gives that $LH=H$. Hence, $v$ satisfies the equation
\begin{equation}
\Delta v=2\frac{|\nabla H|^2}{H^3} - \frac{\Delta H}{H^2} = \frac{1}{2} \lb{x, \nabla v} + 2 \frac{|\nabla v|^2}{v} + \Big(|A|^2 - \frac{1}{2}\Big) v.
\end{equation}

We now fix the function \begin{equation}
h(y)=\frac{y}{1-ky},
\end{equation}
where $k=(2v_0^2)^{-1}$. Simple computations give that 
\begin{equation}
\label{eq:dhddh}
h'(y)=\frac{1}{(1-ky)^2}\,\,\,\,\,\text{and}\,\,\,\,\, h''(y) =\frac{2k}{(1-ky)^3}.
\end{equation}
For convenience, in what follows we will abuse notation slightly and write $h=h(v^2)$, $h'=h'(v^2)$ and so on.

Let $f = |A|^2 h$. Then we have
\begin{equation}
\Delta f= h \Delta |A|^2 +|A|^2 \Delta h +2\lb{\nabla |A|^2, \nabla h}.
\end{equation}
Note that
\begin{equation}
\nabla h= h' \nabla v^2= 2h' v \nabla v,\,\,\,\text{ and }\,\,\,
\Delta h =h' \Delta v^2 +h'' |\nabla v^2|^2.
\end{equation}
Combining this with Lemma \ref{lemma:simons} gives that 
\begin{equation}\label{equ:H_2}
\begin{split}
\Delta f & = h \Big(2|\nabla A|^2 + (1 - 2|A|^2) |A|^2 + \frac{1}{2} \lb{x, \nabla |A|^2}\Big) + |A|^2 \Big( h'' |\nabla v^2|^2 +h'\Delta v^2 \Big) \\ & + 2 \lb{\nabla |A|^2, \nabla h}.
\end{split}
\end{equation}
Now we estimate the right hand side of the equation (\ref{equ:H_2}). First, we have
\begin{equation}
2\lb{\nabla |A|^2, \nabla h}=\frac{\lb{\nabla h, \nabla f}}{h}-|A|^2 \frac{|\nabla h|^2}{h} + 4h'|A|v \lb{\nabla |A|, \nabla v}.
\end{equation}
Using the absorbing inequality gives that
\begin{equation}
4h'|A|v \lb{\nabla |A|, \nabla v} \leq \frac{2(h')^2 |A|^2 v^2 |\nabla v|^2}{h} + 2 h |\nabla |A||^2.
\end{equation}
This implies 
\begin{equation}
2\lb{\nabla |A|^2, \nabla h} \geq \frac{\lb{\nabla h, \nabla f}}{h} - 6 \frac{(h')^2 |A|^2 v^2 |\nabla v|^2}{h} - 2 h |\nabla |A||^2.
\end{equation}
We also have that
\begin{equation}
h''|\nabla v^2|^2 +h' \Delta v^2 =4 h'' v^2 |\nabla v|^2 + h' \Big[ v\lb{x,\nabla v}+4|\nabla v|^2 +  (2|A|^2-1)  v^2 +2 |\nabla v|^2\Big]
\end{equation}
and
\begin{equation}
\frac{1}{2}\lb{x, \nabla f} = \frac{h}{2} \lb{x, \nabla |A|^2} +\frac{|A|^2}{2} \lb{x,\nabla h}= \frac{h}{2} \lb{x, \nabla |A|^2} +|A|^2 h' v\lb{x, \nabla v}.
\end{equation}
Therefore, we obtain that
\begin{equation}\label{equ:H_5}
\begin{split}
\Delta f & \geq \frac{\lb{\nabla h, \nabla f}}{h} + \frac{1}{2}\lb{x, \nabla f} +(1-2|A|^2)|A|^2 h + 2h' v^2 |A|^4-h'v^2 |A|^2  \\ & + \Big[ 4h''v^2 + 6\Big( h'- \frac{(h')^2 v^2}{h} \Big) \Big] |A|^2 |\nabla v|^2.
\end{split}
\end{equation}

Now by the choice of $h$ (compare (\ref{eq:dhddh})), we have
\begin{equation}
h-h' v^2=-kh^2,
\end{equation}
and 
\begin{equation}
4h'' v^2 + 6\Big( h'- \frac{(h')^2 v^2}{h} \Big)=\frac{2k}{(1-kv^2)^2}h.
\end{equation}
Inserting these inequalities into (\ref{equ:H_5}) implies that 
\begin{equation}\label{equ:H}
\Delta f \geq \frac{\lb{\nabla h,\nabla f}}{h} + \frac{1}{2} \lb{x, \nabla f} - f + 2kf^2 +\frac{2k |\nabla v|^2}{(1-kv^2)^2} f.
\end{equation}
Here we used that $h'v^2|A|^2 \leq 2f$.
We will set 
\begin{equation}
a=\frac{\nabla h}{h} \,\,\,\, \text{and}\,\,\,\, d =\frac{2k |\nabla v|^2}{(1-kv^2)^2}.
\end{equation}

\begin{lemma}\label{lemma:cut-off}
Let $x_0 \in \mf{R}^{n+1}$ and $\rho>0$, and set $\phi (x)=(\mu(x))_+^3$, where $(\mu(x))_+ = \max(\mu (x),0)$ and $\mu (x)=\rho^2-|x-x_0|^2$. If $\Sigma^n$ is a shrinker, then on $B_\rho (x_0) \cap \Sigma$ we have
\begin{equation}
\Delta \phi= 24 \mu |(x-x_0)^T|^2 - 6n \mu^2 + 6 \mu^2 H \lb{x-x_0, \mf{n}}.
\end{equation} 
In particular, we have the estimate
\begin{equation}
|\Delta \phi (x)| \leq 24\mu \rho^2 +6n \mu^2 +3\mu^2 \rho |x| \leq (24 + 6n) \rho^4 +3\rho^3 |x|.
\end{equation}
\end{lemma}
\begin{proof}
Since $\nabla \phi =-3(\rho^2-|x-x_0|^2)^2 \nabla |x-x_0|^2=-6\mu^2 (x-x_0)^T$, we have
\begin{equation}
\begin{split}
\Delta \phi & = -6 \,\text{div}(\mu^2 (x-x_0)^T) \\ & =-6 \Big[2\mu \lb{\nabla \mu, (x-x_0)^T}+ \mu^2 \Big(n-\lb{x-x_0,\mf{n}}H\Big)\Big] \\ & = 24 \mu |(x-x_0)^T|^2-6n\mu^2 +6 \mu^2 H \lb{x-x_0, \mf{n}}.
\end{split}
\end{equation}
The second claim follows easily from the shrinker equation and the fact that $\mu \leq \rho^2$.
\end{proof}

We are now ready to prove our main curvature estimate.

\subsection{Proof of  Theorem \ref{thm:curv2}}

Now fix a point $x_0 \in B_{R-1}\cap \Sigma$ and set $\rho = R- |x_0|$. Let $\phi$ be the function defined in Lemma \ref{lemma:cut-off}. We will work on $B_\rho(x_0)\cap \Sigma$. Using (\ref{equ:H}) gives that 
\begin{equation}
\begin{split}
\Delta (\phi f) & =\phi \Delta f + f \Delta \phi +2 \lb{\nabla \phi, \nabla f} \\ & \geq \phi \Big[\lb{a+\frac{x}{2}, \nabla f}-f + 2kf^2 + df\Big] + f \Delta \phi +2\lb{\nabla \phi, \nabla f}.
\end{split}
\end{equation}
Note that 
\begin{equation}
\lb{a,\nabla (\phi f)}=\phi \lb{a, \nabla f} +f\lb{a,\nabla \phi}
\end{equation}
and 
\begin{equation}
\lb{\nabla \phi, \nabla (f\phi)}=f|\nabla \phi|^2 + \phi \lb{\nabla \phi, \nabla f}.
\end{equation}
This implies 
\begin{equation}\label{equ:HH_7}
\begin{split}
\Delta (\phi f) & \geq \lb{a+\frac{x}{2},\nabla (\phi f)} - \lb{a+\frac{x}{2},\nabla \phi} f +\phi \Big[(d-1) f + 2kf^2\Big] \\ & + f \Delta \phi + \frac{2}{\phi} \lb{\nabla \phi, \nabla (f\phi)} -2\frac{|\nabla \phi|^2}{\phi} f.
\end{split}
\end{equation}
Now we set $F(x)=\phi (x) f(x)$ and consider its maximum on $B_\rho (x_0) \cap \Sigma$. Since $F$ vanishes on $\partial B_\rho (x_0) \cap \Sigma$, $F$ achieves its maximum at some point $y_0 \in B_\rho (x_0)\cap \Sigma$. At the point $y_0$, we have 
\begin{equation}
\nabla F(y_0)=0\,\,\, \text{and} \,\,\, \Delta F(y_0)\leq 0.
\end{equation}

In the following, we will work at the point $y_0$. By (\ref{equ:HH_7}) and  $f(y_0)>0$,  we have
\begin{equation}\label{equ:HH_12}
\lb{a + \frac{y_0}{2}, \nabla \phi} + 2\frac{|\nabla \phi|^2}{\phi} \geq \phi (d - 1) + 2k\phi f + \Delta \phi.
\end{equation}
Note that
\begin{equation}
|a|^2 = 4 \Big( \frac{h'}{h} \Big)^2 v^2 |\nabla v|^2 =\frac{2}{k v^2}d \leq \frac{|y_0|^2}{2 k} d.
\end{equation}
This yields that
\begin{equation}\label{equ:HH_15}
\lb{a,\nabla \phi} \leq (d+1) \phi +\frac{|a|^2}{4(d +1) } \frac{|\nabla \phi|^2}{\phi} \leq (d+1) \phi + \frac{|y_0|^2}{8 k} \frac{|\nabla \phi|^2}{\phi}.
\end{equation}
Combining (\ref{equ:HH_15}) with (\ref{equ:HH_12}) gives that 
\begin{equation}\label{equ:HH_17}
2 k \phi f \leq -\Delta \phi +2\phi + \Big(2 +\frac{|y_0|^2}{8 k} \Big)\frac{|\nabla \phi|^2}{\phi} +\frac{|y_0|}{2}  |\nabla \phi|.
\end{equation}
By the definition of $\phi$, we have
\begin{equation}
\phi \leq \rho^6,\,\,\,\, |\nabla \phi| \leq 6\rho^5\,\,\,\,\text{ and }\,\,\,\, \frac{|\nabla \phi|^2}{\phi}\leq 36 \rho^4.
\end{equation}
Combining this with Lemma \ref{lemma:cut-off}, $|y_0|\leq R$ and (\ref{equ:HH_17}) yields that
\begin{equation}
F(y_0)=\phi (y_0)f(y_0)\leq C (\rho^6 + R \rho^5 + R^2 \rho^4),
\end{equation} 
where $C$ is a constant depending on $n$ and $\delta$.

Since $F$ achieves its maximum at $y_0$, we have $F(x_0) \leq F(y_0)$. This implies 
\begin{equation}
\frac{\rho^6 |A|^2 (x_0)}{H^2(x_0) - k} = F(x_0) \leq F(y_0) \leq C (\rho^6 + R \rho^5 + R^2 \rho^4).
\end{equation}
In particular, we have
\begin{equation}
|A|(x_0) \leq C  \Big(1 + \frac{R}{\rho} \Big)H(x_0).
\end{equation}
Since $x_0$ is an arbitrary point in $B_{R-1}\cap \Sigma$, this completes the proof of Theorem \ref{thm:curv2}.

\section{Rigidity theorems for mean convex shrinkers}
\label{sec:rigidity}

In this section we prove Theorems \ref{thm:H>0} and \ref{thm:main} by adapting the iteration and improvement scheme used to prove \cite[Theorem 0.1]{CM5}. For convenience of the reader, we briefly outline this scheme here; recall that the two key ingredients are the so-called iterative step \cite[Proposition 2.1]{CM5} and the improvement step \cite[Proposition 2.2]{CM5} (compare Proposition \ref{prop:v:u:improve} below). In the iterative step, it is shown that if a self-shrinker is almost cylindrical (quantified by $H$ and $|A|$) on a large scale, then it is still close to a cylinder on a larger scale, albeit with some loss in the estimates. It is important here that the scale extends by a fixed multiplicative factor.

\begin{proposition}(Iteration; \cite[Proposition 2.1]{CM5}) \label{prop:v:u:step1} Given $\lambda_0 < 2$ and $n$, there exist
  positive constants $R_0$, $\delta_0$, $C_0$ and $\theta$ so that if $\Sigma^n \subset \mf{R}^{n+1}$ is a shrinker with
  $\lambda (\Sigma) \leq \lambda_0$, $R \geq R_0$, and
  \begin{itemize}
  \item $B_R \cap \Sigma$ is smooth with $H\geq 1/4$ and $|A| \leq
    2$,
  \end{itemize}
  then $B_{(1+\theta)R} \cap \Sigma$ is smooth with $H \geq \delta_0 $
  and $|A| \leq C_0$.
\end{proposition}

On the other hand, in the improvement step, it is shown that if a shrinker is close to a cylinder on some scale, then the estimates can be improved so long as we decrease the scale by a fixed amount. We will show that the initial closeness in the improvement step only needs to be quantified by $H$ --- using our curvature estimate Theorem \ref{thm:curv2} and a compactness result of shrinkers, we can show that the bounded curvature assumption in the improvement step (Proposition 2.2 of \cite{CM5}) can be removed, which in turn implies Theorem \ref{thm:H>0}. Our improvement step is stated as follows:

\begin{proposition}(Improvement) \label{prop:v:u:improve} 
Given $n$ and $\lambda_0$, let $\delta_0 \in (0,1/4)$ be given by Proposition \ref{prop:v:u:step1}. Then there exists $R=R(n,\lambda_0)$ so that if $\Sigma^n \subset \mf{R}^{n+1}$ is a shrinker with $\lambda (\Sigma) \leq
  \lambda_0$ and
  \begin{itemize}
  \item $H\geq \delta_0$ on $B_{R} \cap \Sigma$,
  \end{itemize}
  then $H\geq 1/4$ and $|A| \leq 2$ on $B_{R-4} \cap \Sigma$.
\end{proposition}

The main argument in the improvement step is to control the derivatives of the tensor $\tau = A/H$. These estimates are shown to decay exponentially as $R^\alpha \ee^{-R/4}$ for some $\alpha$, allowing one to extend good cylindrical estimates from a fixed scale $5\sqrt{2n}$ to almost the whole ball of radius $R$. For us, instead of assuming $|A|\leq C$ for some constant $C$ as in \cite{CM5}, our curvature estimates give that $|A|\leq C R$ for shrinkers with positive mean curvature $H$ in $B_R$. In the proof of Proposition \ref{prop:v:u:improve}, we show that this is still enough to control the derivatives of $\tau$, possibly with a worse exponent $\alpha$ of $R$. But the exponential factor still decays much faster than any polynomial factor, so the polynomial factor can be eventually absorbed into the exponential factor as long as we choose $R$ sufficiently large. The remaining details of our proof will be deferred to Section \ref{sec:improvement}.\\ 

To complete the iteration and improvement scheme, we first apply Proposition \ref{prop:v:u:improve}, then apply Proposition \ref{prop:v:u:step1} and repeat the process. The multiplicative factor extends the scale by more than the fixed decrease if $R$ is large enough, so we get strict mean convexity on all of $\Sigma$, which must therefore be a cylinder by the classification of mean convex shrinkers (Theorem \ref{thm:huisken}). Thus we have:
\begin{proposition}\label{prop:v:u:entropy2} 
Given $n$ and $\lambda_0<2$, let $\delta_0 \in (0,1/4)$ be given by Proposition \ref{prop:v:u:step1}.
Then there exists $R=R(n,\lambda_0)$ so that if $\Sigma^n\subset \mf{R}^{n+1}$ is a shrinker with entropy $\lambda (\Sigma) \leq \lambda_0$  which satisfies
\begin{itemize}
\item $ H \geq \delta_0 $ on $ B_R \cap \Sigma$, 
\end{itemize}
then $\Sigma$ is a generalized cylinder $\mf{S}^k \times \mf{R}^{n-k}$ for some $1\leq k \leq n$.
\end{proposition}

We also need the following compactness theorem for self-shrinkers which plays an important role in our argument: 

\begin{lemma}[Compactness]\label{lemma:v:u:cmpt}
Let $\Sigma_i \subset \mf{R}^{n+1}$ be a sequence of shrinkers with $\lambda(\Sigma_i) \leq \lambda_0$ and 
\begin{equation}
|A| (x) \leq C (1+ |x|)\,\,\, \text{ on }\,\,\, B_i\cap \Sigma_i.
\end{equation}
Then there exists a subsequence $\Sigma_i'$
  that converges smoothly and with multiplicity one to a complete
  embedded shrinker $\Sigma$ with 
 \begin{align} |A| (x) \leq C (1+ |x|)\,\, \text{ and }\,\,\, \lim_{i \to \infty} \, \lambda
    (\Sigma_i') = \lambda (\Sigma) \, .
  \end{align}  
  \end{lemma}
  \begin{proof}
The key is that the a priori bound on $|A|$ is uniform on compact subsets. Thus, as in Lemma 2.7 in \cite{CM5}, for any $R$ we may obtain smooth convergence in $B_R$ by covering with a finite number of balls. Passing to a diagonal argument gives the overall smooth convergence to a smooth, complete, embedded shrinker $\Sigma$ with $\lambda(\Sigma)\leq \lambda_0$. Again arguing as in \cite{CM5}, if multiplicity is greater than one then the limit $\Sigma$ must be $L$-stable. But there are no such shrinkers with polynomial volume growth (see Theorem 0.5 in \cite{CM2}), so the multiplicity must be one.
\end{proof}



Now we are ready to prove Theorem \ref{thm:main}.

\begin{proof}[Proof of Theorem \ref{thm:main}]
Since we assumed $H\geq \delta$ on $B_R\cap \Sigma$, the curvature estimate Theorem \ref{thm:curv2} gives in particular that $|A| \leq CH \leq \frac{C}{2}|x|$ on $B_{R/2}\cap \Sigma$. Applying the compactness Lemma \ref{lemma:v:u:cmpt} we get that $\Sigma$ is smoothly close to $\mf{S}^k \times \mf{R}^{n-k}$ in $B_{R/2}$. Thus for $R$ sufficiently large we may assume $\lambda(\Sigma) \leq \lambda_0<2$, and $H\geq \delta_0$ on $B_{R/2}\cap \Sigma$. The result then follows from Proposition \ref{prop:v:u:entropy2}. 
\end{proof}

\subsection{Proof of Theorem \ref{thm:H>0}}

For the proof of Theorem \ref{thm:H>0}, we will need the following curvature estimate from Section 3 in \cite{GZ4} (see in particular Theorem 0.4 and Remark 3.6 therein). The key fact was that $LH=H$ on any shrinker $\Sigma$, which implies an almost-stability inequality for $\Sigma$ if the eigenfunction $H$ is positive. 

\begin{lemma}\label{lemma:v:u:positive_mean}
Given $n\leq 6$ and $\alpha>0$, there exists $C=C(n,\alpha)$ so that if $\Sigma^n \subset \mf{R}^{n+1}$ is a shrinker with $\lambda(\Sigma)\leq \alpha$ and $H>0$ on $B_R \cap \Sigma$ for some $R > 2$, then on $B_{R-1}\cap \Sigma$ we have
\begin{equation}
|A| \leq C(1+|x|).
\end{equation}
\end{lemma}

\begin{remark}Note that for properly embedded self-shrinkers with finite genus in $\mf{R}^3$, Song \cite{Song14} (see also \cite{WL2}) gave the linear growth of the second fundamental form.
\end{remark}

Now we give the proof of Theorem \ref{thm:H>0} via Proposition \ref{prop:v:u:entropy2}.
\begin{proof}[Proof of Theorem \ref{thm:H>0}  using Proposition \ref{prop:v:u:entropy2}]
First, the Harnack inequality gives that either $H\equiv 0$ or $H>0$. If $H\equiv 0$ in $B_R$, then $\Sigma$ is a hyperplane in $B_R$. Thus by the rigidity of the hyperplane (for example, Theorem 0.1 in \cite{GZ4} or Theorem \ref{thm:graph}, or even directly by Brakke's theorem \cite{BK}), $\Sigma$ must be a hyperplane $\mf{R}^n$ if $R$ is sufficiently large. 

Next, we assume $H>0$ in $B_R$. Lemma \ref{lemma:v:u:positive_mean} then gives a curvature estimate on $B_{R-1}\cap \Sigma$. By the compactness of Lemma \ref{lemma:v:u:cmpt}, we can assume that $\Sigma$ is smoothly close to $\mf{S}^k \times \mf{R}^{n-k}$ in $B_{R_1}$ for some $k\geq 0$, where $R_1$ can be taken as large as we wish. If $k=0$, then again the rigidity of the hyperplane means that $\Sigma$ must be a hyperplane, although this is a contradiction since in this case we assume $H>0$ on $B_R\cap \Sigma$. So $k\geq1$, and consequently $H$ is approximately $\sqrt{k/2} $ on $B_{R_1}\cap \Sigma$, then Theorem \ref{thm:H>0} follows directly from Proposition \ref{prop:v:u:entropy2}. 
\end{proof}

\begin{remark}
\label{rmk:direct}
In the above proofs of Theorems \ref{thm:H>0} and \ref{thm:main}, the smooth closeness (obtained via compactness) also implies a bound for $|A|$ on a large ball, so at that point we could also appeal directly to Theorem 0.1 in \cite{CM5}. The compactness Lemma \ref{lemma:v:u:cmpt} can also give a shorter proof of our main rigidity theorem for graphical shrinkers in \cite{GZ4}, but in both cases we feel that the more effective proofs given may provide a more complete understanding.
\end{remark}

\subsection{Proof of the improvement step}
\label{sec:improvement}
In this subsection, we prove Proposition \ref{prop:v:u:improve} by sketching the necessary modifications of the proof of Proposition 2.2 in \cite{CM5}.

As discussed earlier, the central argument is the very tight estimate on the tensor $\tau=A/H$, that decays exponentially in $R$. Thus, our main modification is the following lemma, which removes the curvature bound of Corollary 4.12 in \cite{CM5} by accepting a slightly larger power of $R$, although we still have the exponential decay. 

\begin{lemma} 
  \label{cor:v:u:taubound_1} 
  Given $n$, $\lambda_0$ and $\delta>0$, there exists a
  constant $C_{\tau}>0$ such that if $\lambda (\Sigma) \leq
  \lambda_0$, $R\geq 2$, and
  \begin{itemize}
  \item $B_{R+1} \cap \Sigma$ is smooth with  $H\geq
    \delta>0$,
  \end{itemize}
  then 
  \begin{align}
    \sup_{ B_{R-2} \cap \Sigma } \, \, \left| \nabla \tau \right|^2 +
    R^{-4} \, \left| \nabla^2 \tau \right|^2 \leq C_{\tau} \, R^{3n+4}
    \, \ee^{-R/4} \, .
  \end{align}
\end{lemma}
\begin{proof}
First, Theorem \ref{thm:curv2} gives there exists a constant $C=C(n, \delta)$ such that $|A|\leq C R H$ in $B_R$. 
Hence, Proposition 4.8 in \cite{CM5} with $s=1/2$ implies that 
\begin{align}
    \int_{B_{R-1/2} \cap \Sigma} \left| \nabla \tau \right|^2 \, \ee^{-
      |x|^2/4 } &\leq C \, R^{n+4} \, \ee^{ - (R-1/2)^2/4} \, .
  \end{align}
  Since $\ee^{- |x|^2/4 } \geq \ee^{- \frac{R^2 -2R + 1}{4} } $ on
  $B_{R-1}$, it follows that
  \begin{align}\label{equ:cor:integral}
    \int_{B_{R-1} \cap \Sigma} \left|\nabla \tau \right|^2 &\leq C \,
    R^{n + 4 } \, \ee^{ - \frac{R}{4} } \, .
  \end{align}
  This gives the desired integral decay on $\nabla \tau$.  We will
  combine this with elliptic theory to get the pointwise bounds. The
  key is that $\tau$ satisfies the elliptic equation $\mathcal{L}_{H^2} \tau =0$ (see Proposition 4.5 in \cite{CM5}), that is,
  \begin{equation}\label{equ:cor:tau}
  \Delta \tau -\frac{1}{2}\lb{x, \nabla \tau} + \lb{\nabla \log H^2, \nabla \tau}=0.
  \end{equation}
Note that we have 
\begin{equation}
|\nabla \log H^2| = \frac{2|\nabla H|}{H} \leq \frac{|A||x|}{H}\leq C R |x|,
\end{equation}  
where we used that $|\nabla H| \leq \frac{1}{2}|A||x|$ and $|A| \leq C R H$. 

Therefore, the two first order terms in the equation (\ref{equ:cor:tau}) come from $x^T$ in
  $\mathcal{L}$ and $\nabla \log H^2$; both grow at most quadratically. Now we can apply elliptic theory on balls of radius $1/R^2$ to
  get for any $p \in B_{R-2} \cap \Sigma$ that
  \begin{align}
    \left( |\nabla \tau|^2 + R^{-4} \, | \nabla^2 \tau |^2 \right)(p)
    \leq C \, R^{2n} \, \int_{B_{\frac{1}{R^2} } (p) \cap \Sigma} | \nabla
    \tau |^2 \, .
  \end{align}
Combining this with the integral bounds (\ref{equ:cor:integral}) gives the lemma.
\end{proof}
Now we sketch the proof of Proposition \ref{prop:v:u:improve}. 

Fix $n$, $\lambda_0>0$, and $\delta_0>0$. Let $R>0$ and
assume that $\Sigma$ is a self-shrinker in $\mf{R}^{n+1}$, $\lambda(\Sigma)\le\lambda_0$ and $H \geq \delta_0$ on $\Sigma\cap B_R$. By Lemma \ref{cor:v:u:taubound_1}, the tensor $\tau= A/H$ satisfies 
\begin{equation}
\left| \nabla \tau \right| + \left|\nabla^2 \tau \right| 
\leq \varepsilon_\tau  \quad \text{ on }B_{R-2} \cap \Sigma,
\end{equation}
where
\begin{align*}
  \varepsilon_{\tau}^2 := C \, R^{3n+8} \, \ee^{-R/4} 
\end{align*}
and the constant $C$ depends only on $n$, $\delta_0$ and $\lambda_0$. As in \cite{CM5}, the key point is that $\epsilon_\tau$ can still be made small for large $R$, due to the decaying exponential factor.

Now fix  small $\varepsilon_0 > 0$, to be chosen as needed, but depending only on $n$. Combining the compactness of Lemma \ref{lemma:v:u:cmpt} with the classification of mean convex shrinkers \cite{CM1}, there exists a constant $R_1=R_1(n, \lambda_0, \delta_0,\varepsilon_0)$ so that if $R \geq R_1$, 
then $B_{5\sqrt{2n}} \cap \Sigma$ is $C^2$ $\varepsilon_0$-close to a cylinder $\mf{S}^k \times \mf{R}^{n-k}$ for some $1\leq k \leq n$. The remainder of Proposition \ref{prop:v:u:improve} follows from the proof of Proposition 2.2 in \cite{CM5}.

\section{Bernstein type theorems}
\label{sec:bernstein}
\subsection{Rigidity of the hyperplane self-shrinker}

In this subsection, we will prove Theorem \ref{thm:graph} which gives a Bernstein type theorem for self-shrinkers in all dimensions. The key is that the positive lower bound of $w=\lb{V, \mf{n}}$ enables us to  obtain a curvature estimate in all dimensions. This is the content of the next theorem.

\begin{theorem}\label{thm:curv_graph}
Given $n$ and $\delta>0$, there exists $C=C(n, \delta)$ so that for any smooth properly embedded self-shrinker $\Sigma^n \subset \mf{R}^{n+1}$  
which satisfies
\begin{itemize}
 \item $w=\lb{V,\mf{n}}\geq \delta$ on $B_R \cap \Sigma$ for some constant unit vector $V$ and $R>2$,
\end{itemize}
we have
\begin{equation}
|A| \leq C, \,\,\,\,\,\,\, \text{on } B_{R/2}\cap \Sigma.
\end{equation}
\end{theorem}

Theorem \ref{thm:curv_graph} is essentially a corollary of Theorem 3.1 in \cite{EH2}, and the proof is similar to Theorem \ref{thm:curv2} --- the essential component being that $Lw=\frac{1}{2}w$ (see Lemma \ref{lemma:simons}). Combining Theorem \ref{thm:curv_graph} and some ingredients from \cite{GZ4}, we can now prove Theorem \ref{thm:graph}.

\begin{proof}[Proof of Theorem \ref{thm:graph}]
Given $n$, $\lambda_0$ and $\delta$, Theorem \ref{thm:curv_graph} gives a curvature bound $C$. Since   $\Sigma$ is graphical and satisfies a curvature bound, Theorem 2.2 in \cite{GZ4} allows us to make $|A|$ as small as we want  by choosing $R$ sufficiently large. In particular, we can choose $R$ such that $|A|^2 \leq 1/4$ on $B_{R/2}\cap \Sigma$. Now Theorem \ref{thm:graph} follows directly from the compactness of Lemma \ref{lemma:v:u:cmpt}, Brakke's Theorem \cite{BK} (see also \cite{white}) and the fact that any complete shrinker with $|A|^2 <1/2$ is a hyperplane (see \cite{CL}). 
\end{proof}

\subsection{Curvature estimates and a Bernstein type theorem for translators}
In this subsection, we will sketch that the methods used in Section \ref{sec:curv} can also be applied to prove a curvature estimate for translators with positive lower bound for the normal part of a constant vector field. As an immediate corollary, we obtain the Bernstein type theorem for translators which was proved by Bao and Shi  \cite{BS} by using different methods.

Recall that a smooth hypersurface $\Sigma^{n} \subset \mathbf{R}^{n+1}$ is called a \textit{translating soliton}, or \textit{translator} for short, if it satisfies the equation
\begin{equation}
H=-\lb{y,\mathbf{n}},
\end{equation}
where $y\in \mathbf{R}^{n+1}$ is a constant vector. For simplicity, we may assume $y=e_{n+1}$, so that translators satisfy the equation
\begin{equation}
H=-\lb{e_{n+1},\mathbf{n}}.
\end{equation}

The curvature estimate for translators is the following:
\begin{theorem}\label{thm:curv_t}
Given $n$ and $\delta>0$, there exists $C=C(n, \delta)$ so that for any smooth properly embedded translator $\Sigma^n \subset \mf{R}^{n+1}$  
which satisfies
\begin{itemize}
\item $w=\lb{V,\mf{n}} \geq \delta$ on $ B_R(x_0) \cap \Sigma$ for some constant unit vector $V$ and  $x_0 \in \mf{R}^{n+1}$,
\end{itemize}
we have
\begin{equation}
|A|^2(x) \leq C \Big(\frac{1}{R} + \frac{1}{R^2}\Big) w^2(x),\,\,\,\,\,\,\, \text{for all}\,\,\,x\in B_{R/2}(x_0)\cap \Sigma.
\end{equation}
\end{theorem}
\begin{proof}
Since the proof is very similar to Theorem \ref{thm:curv2}, we will only sketch the argument.

The stability operator $\mathfrak{L}$ for translators is defined by $
\mathfrak{L}= \Delta + \lb{e_{n+1}, \nabla \cdot} + |A|^2$. We have the following identities (see for instance \cite{IR1})
\begin{equation}
\mathfrak{L} \lb{V,\mf{n}} =0\,\,\, \text{ and }\,\,\
\mathfrak{L} |A|^2 = 2 |\nabla A|^2 - |A|^2.
\end{equation}
Set $v=1/w$, $v_0 = 1/\delta$ and $f = |A|^2 h$. Similar computations and estimates as in the proof of Theorem \ref{thm:curv2} give that 
\begin{equation}
\begin{split}
\Delta f & \geq \frac{\lb{\nabla h, \nabla f}}{h} - \lb{e_{n+1}, \nabla f} -2h|A|^4 + 2h' v^2 |A|^4  \\ & + \Big[ 4h''v^2 + 6\Big( h'- \frac{(h')^2 v^2}{h} \Big) \Big] |A|^2 |\nabla v|^2.
\end{split}
\end{equation}
Choosing $h(y) =\frac{y}{1-ky}$, where $k=(2v_0^2)^{-1}$. We then obtain that
\begin{equation}\label{equ:H_t}
\Delta f \geq \frac{\lb{\nabla h,\nabla f}}{h} - \lb{e_{n+1}, \nabla f} + 2kf^2 +\frac{2k |\nabla v|^2}{(1-kv^2)^2} f.
\end{equation}
Let $\phi(x)=((R^2-|x-x_0|^2)_+)^3$. We set $F(x)=\phi (x) f(x)$ and consider its maximum on $B_R(x_0) \cap \Sigma$. Assume $F$ achieves its maximum at some point $y_0 \in B_R(x_0)\cap \Sigma$. 

Using $\nabla F(y_0)=0$, $\Delta F(y_0)\leq 0$ and some estimates of $\phi$, we  have 
\begin{equation}
F(y_0)=\phi (y_0)f(y_0)\leq C( R^4 + R^5 ),
\end{equation} 
where $C$ is a constant depending on $n$ and $\delta$.

Since $F$ achieves its maximum at $y_0$, we have $F(x) \leq F(y_0)$ for all $x\in B_{R/2}(x_0)\cap \Sigma$. This implies for any $x\in B_{R/2}(x_0)\cap \Sigma$
\begin{equation}
\Big(\frac{R}{2}\Big)^6 \frac{ |A|^2 (x)}{w^2(x) - k} \leq F(x) \leq F(y_0) \leq C ( R^4 + R^5 ).
\end{equation}
Now the theorem follows directly. 
\end{proof}

By taking $R$ goes to infinity in Theorem \ref{thm:curv_t}, we obtain the Bernstein type theorem for translators in \cite{BS}.
\begin{theorem}(\cite{BS})\label{thm:graph_t}
Let  $\Sigma^n \subset \mf{R}^{n+1}$ be a smooth complete translator. If there exists a positive constant $\delta$ such that 
\begin{itemize}
\item $w=\lb{V,\mf{n}} \geq \delta $ for some constant unit vector $V$,
\end{itemize}
then $\Sigma$ must be a hyperplane.
\end{theorem}

Recently, Kunikawa \cite{KK2} generalized Theorem \ref{thm:graph_t} to arbitrary codimension.

\bibliographystyle{alpha}
\bibliography{Rigidity_new2}
\end{document}